\documentclass[11pt,a4paper]{amsart}
\usepackage{amsfonts,amsmath,amssymb,amsthm}
\usepackage{times}
\usepackage{bbm}

\usepackage{hyperref}
\hyperbaseurl{.}

\numberwithin{equation}{section}

\renewcommand{\epsilon}{\varepsilon}

\DeclareSymbolFont{SY}{U}{psy}{m}{n}
\DeclareMathSymbol{\emptyset}{\mathord}{SY}{'306}

\DeclareMathOperator{\ran}{Ran} \DeclareMathOperator{\Ker}{Ker}
 
\DeclareMathOperator{\sign}{sign} 
\DeclareMathOperator{\dom}{Dom}

\DeclareMathOperator{\sgn}{sgn}

\DeclareMathSymbol{\newtimes}{\mathbin}{SY}{'264}

\newcommand{\T}{\mathbb{T}}

\newcommand{\1}{\mathbb{I}}

\newcommand{\fL}{\mathfrak{L}}

\newcommand{\fa}{\mathfrak{a}}
\newcommand{\fb}{\mathfrak{b}}

\newcommand{\fh}{\mathfrak{h}}
\newcommand{\ft}{\mathfrak{t}}

\newcommand{\fv}{\mathfrak{v}}
\newcommand{\cH}{{\mathcal H}}

\newcommand{\cL}{{\mathcal L}}

\newtheorem{introtheorem}{Theorem}{\bf}{\it}
{\bf}{\it}
{\bf}{\it}
{\bf}{\it}
{\bf}{\it}

{\bf}{\it}

\newtheorem{theorem}{Theorem}[section]{\bf}{\it}
{\bf}{\it}
\newtheorem{corollary}[theorem]{Corollary}{\bf}{\it}
\newtheorem{example}[theorem]{Example}{\it}{\rm}
\newtheorem{lemma}[theorem]{Lemma}{\bf}{\it}
\newtheorem{remark}[theorem]{Remark}{\it}{\rm}
\newtheorem{definition}[theorem]{Definition}{\bf}{\it}
{\bf}{\it}
{\bf}{\it}
{\bf}{\it}
\newtheorem{hypothesis}[theorem]{Hypothesis}{\bf}{\it}

\title[Representation Theorems]{Representation Theorems for indefinite quadratic forms without spectral gap}

\author[S.\ Schmitz]{Stephan Schmitz$^*$}
\address{S.~Schmitz, FB 08 - Institut f\"{u}r Mathematik\\
Johannes Gutenberg-Universit\"{a}t Mainz\\\newline
Stau\-dinger Weg 9,
D-55099 Mainz,
Germany}
\email{schmist@uni-mainz.de}
\thanks{* The material presented here is part of the author's Ph.~D.~ thesis}

\date{\today}

\begin{document}

\begin{abstract}
The First and Second Representation Theorem for sign-indefinite quadratic forms are extended.
We include new cases of unbounded forms associated with operators that do not necessarily have a spectral gap around zero.
The kernel of the associated operators is determined for special cases. This extends results by Grubi\v{s}i\'c, Kostrykin, Makarov and Veseli\'c in 
[Mathematika \textbf{59} (2013), 169 -- 189].
\end{abstract} 

\maketitle

\section{Introduction and main results}
%%%%%%%%%%%%%%%%%%%%%%%%%%%%%%%%%%%%%%%%%%%%%%%%%%%%%%%%%%%%%
In this work, we consider the Representation Theorems for symmetric sesquilinear forms in a Hilbert space \(\cH\). Before we introduce these Theorems, we fix the following notions.

A form \(\fb\) is said to be \emph{associated} with a self-adjoint operator \(B\) if 
\begin{equation}\label{eq:1repintro}\fb[x,y]=\langle x,By\rangle\quad \text{for all } x\in \dom[\fb],\;y\in \dom(B)\subseteq \dom[\fb].\end{equation}
If, in addition, the domain stability condition \begin{equation}\label{introdom}\dom(|B|^{1/2})=\dom[\fb]\end{equation} holds, then the form is said to be \textit{represented} by the operator \(B\), that is,  
\begin{equation}\label{eq:2repintro}\fb[x,y]=\langle |B|^{1/2}x,\sign(B)|B|^{1/2}y\rangle  \quad \text{for all }x,y\in \dom[\fb].\end{equation}

The representations \eqref{eq:1repintro} and \eqref{eq:2repintro} are usually called the First and Second Representation Theorem, respectively. 
Taken together, the Representation Theorems give the one-to-one correspondence of forms and operators.
These Representation Theorems however do not hold for arbitrary forms.

For bounded forms, the Representation Theorems hold true by the Riesz Representation Theorem.
Classical results verify the Representation Theorems for closed semibounded forms, see, e.g., \cite[Section VI.2]{K}.

For forms that are \textit{indefinite}, that is non-semibounded, we cannot expect that the Representation Theorems hold in general. As a consequence the correspondence between forms and operators is more complicated. It may for instance happen that two (or even infinitely many) forms define the same
self-adjoint operator B but this operator defines only one of these forms by means
of \eqref{eq:2repintro}, see, e.g. \cite[Example 2.11 and Proposition 4.2]{GKMV}.
In this case  \(\dom(|B|^{1/2})\neq \dom[\fb]\) and only the First but not the Second Representation Theorem holds.  

In \cite{GKMV}, the Representation Theorems are verified in a special case with new proofs. Namely, for indefinite forms  \(\fb\) of the type 
\begin{equation}\label{intro:b}\fb[x,y]=\langle A^{1/2}x,HA^{1/2}y\rangle,\;x,y\in \dom[\fb]=\dom(A^{1/2}),\end{equation}  where \(A\) and \(H\) are self-adjoint operators such that \(A\geq cI>0\) and \(H\) is bounded with \(\|H^{-1}\|\leq \alpha^{-1}<\infty\), the following theorem holds:

\textbf{Theorem.} \textit{Let \(\fb\) be given by \eqref{intro:b}. Then there exists a unique self-adjoint, boundedly invertible operator \(B=A^{1/2}HA^{1/2}\) with \((-\alpha c,\alpha c)\subset \rho(B)\) associated with the form \(\fb\). If in addition \(\dom(|B|^{1/2})=\dom(A^{1/2})\) holds, then the form \(\fb\) is represented by the operator \(B\).}

\vspace{6pt}

For this theorem to hold, the strict positivity of \(A\) is crucial since it allows to construct the bounded self-adjoint operator \(A^{-1/2}H^{-1}A^{-1/2}\) which is inverse to the operator \(B\).  

It is now a natural project to drop the strict positivity condition on \(A\) in this theorem. However, Example \ref{counterexample} below shows that the operator \(B\) may be non-closed if this condition is dropped.
Under additional assumptions however, the closedness and thus the self-adjointness of the operator  \(B=A^{1/2}HA^{1/2}\) can be preserved, so that the Representation Theorems can be extended.
\vspace{6pt}

In the present work, we assume \(H^{-1}\) to be bounded but assume \(A\) to be only non-negative. Our main results on the Representation Theorems in this work are the following: 
\begin{introtheorem}
Let \(A\geq0\) be self-adjoint and let \(H,H^{-1}\) be bounded, self-adjoint. Assume that there is a self-adjoint involution \(J_A\) commuting with \(A\) such that
 \[P_AHP_A\geq \alpha P_A,\quad P^\bot_AHP_A^\bot\leq -\alpha P_A^\bot\] for some \(\alpha \in(0,1] \), where \(P_A:=\frac{1}{2}(J_A+I)\) and \(P_A^\bot\) is the complementary projector. 

Then, there exists a unique self-adjoint operator \(B=A^{1/2}HA^{1/2}\) with \[\fb[x,y]=\langle x,By\rangle\quad \text{for all } x\in \dom[\fb],\;y\in \dom(B)\subseteq \dom[\fb].\] If additionally the domain stability condition \eqref{introdom} holds, then
\[\fb[x,y]=\langle |B|^{1/2}x,\sign(B)|B|^{1/2}y\rangle  \quad \text{for all }x,y\in \dom[\fb].\]
\end{introtheorem}

In particular, we obtain the following result for forms \(\fb=\fa+\fv\) consisting of a diagonal part \(\fa\) and an off-diagonal part \(\fv\).

\begin{introtheorem}
Let \(A\geq 0\) be a self-adjoint operator, \(J_A\) a self-adjoint involution commuting with \(A\) and  \[\fa[x,y]:=\langle A^{1/2}x,J_AA^{1/2}y\rangle,\quad \dom[\fa]=\dom(A^{1/2}).\] Suppose that \(\fv\) is a symmetric form and \(\beta\) is a finite constant with
\[\dom[\fv]\supseteq\dom[\fa],\quad \fv[J_Ax,y]=-\fv[x,J_Ay], \quad|\fv[x]|\leq \beta\big\|(A+I)^{1/2}x\big\|^2,\;x,y\in \dom[\fa].\]
Then the form \(\fb:=\fa+\fv\) is associated with a unique self-adjoint operator \(B\). If additionally \eqref{introdom} holds, then \(\fb\) is represented by the operator \(B\). 
The kernel of the operator \(B\) can explicitly be written as 
\[\Ker(B)=(\Ker(A_+)\cap \cL_+)\oplus(\Ker(A_-)\cap \cL_-),\] where \(A=A_+\oplus A_-\) is the decomposition induced by \(J_A\) and 
\[\cL_\pm:=\big\{x_\pm\in \dom(A_\pm^{1/2})\;| \; \fv[x_+\oplus0,0\oplus x_-]=0 \quad \text{for all }x_\mp\in \dom(A_\mp^{1/2})\big\}.\] 
\end{introtheorem}

The general idea behind the proofs of the Representation Theorems is to consider the perturbed form \(\tilde{\fb}=\fb+J_A\) that is in the framework of  \cite{GKMV}, and to pull back the results to the unperturbed form \(\fb\). 

The domain stability condition \eqref{introdom} is in general hard to verify directly. In \cite{GKMV}, equivalent statements as well as sufficient conditions for \eqref{introdom} are given. These statements can be generalised to the situation here. Namely, the inclusions \(\dom(A^{1/2})\subseteq\dom(|B|^{1/2})\), \(\dom(A^{1/2})\subseteq\dom(|B|^{1/2})\), and \(\sgn(B)\dom(A^{1/2})\subseteq \dom(A^{1/2})\) are equivalent to each other for any choice of the value \(\sgn(0)\in \{-1,1\}\) of the unitary sign-function. Furthermore, in each of the three cases \(H\dom(A^{1/2})\subseteq \dom(A^{1/2})\), \(H\geq cI>0\) and \(B\) semibounded, respectively, the domain stability condition  \eqref{introdom} is satisfied, see Section \ref{secDom} below.
\vspace{6pt}
 
The paper is organized as follows: In Section 2, we prove the First Representation Theorem in the general setting (Theorem \ref{1repgeneral}) and for off-diagonal perturbations of indefinite diagonal forms (Theorem \ref{1repoffdiag}). In the setting of Theorem \ref{1repoffdiag}, we give a characterisation for the kernel of the associated operator (Theorem \ref{thm:incl}). 

Section 3 is devoted to the Second Representation Theorem (Theorem \ref{2rep}) and the analysis of the domain stability condition \eqref{introdom}.   
We give statements equivalent to the domain stability condition as well as sufficient criteria , Theorem \ref{equivalent} and Lemma \ref{sufficient}, respectively.

Theorem 1 is a combination of Theorems {\ref{1repgeneral}} and {\ref{2rep}}. Theorem 2 can be derived from Theorems {\ref{1repoffdiag}}, {\ref{2rep}} and {\ref{thm:incl}}.

We use the following notation: The domain  and range of a densely defined operator \(T\) on a Hilbert space \(\cH\) are denoted by \(\dom(T)\) and \(\ran(T)\). The adjoint operator of \(T\) and its resolvent set  are denoted by \(T^*\) and \(\rho(T)\).
For forms \(\fb\), we denote the domain by \(\dom[\fb]\) and write \(\fb[x]:=\fb[x,x]\) for shortness. 
The inner product, norm and identity operator on \(\cH\) are  
denoted by \( \langle \;\cdot\;,\;\cdot\;\rangle_{\cH}\), \(\left\|\;\cdot\;\right\|_{\cH}\) and \(I_{\cH}\) respectively, where the subscript \(\cH\) is omitted if no confusion can arise. Finally, if \(P\) is an orthogonal projector, we write \(P^\bot=I-P\) for the complementary projector. 

\subsection*{Acknowledgements.}
The author would like to thank his Ph.\ D.\ advisor Vadim Kostrykin for the introduction to this field of research and many useful discussions.
The author would also like to thank Albrecht Seelmann and Amru Hussein for helpful remarks on the manuscript.

%%%%%%%%%%%%%%%%%%%%%%%%%%%%%%%%%%%%%%%%%%
\section{The First Representation Theorem}
%%%%%%%%%%%%%%%%%%%%%%%%%%%%%%%%%%%%%%%%%%

\subsection{The general case} 
To start, we fix the following assumptions. 

\begin{hypothesis}\label{assumption1}

Let \(A\) and \(H\) be self-adjoint operators in \(\cH\) such that \(A\) is non-negative and \(H\) is bounded and boundedly invertible.
Furthermore, suppose that there exists a self-adjoint involution \(J_A\neq \pm I\) commuting with \(A\) such that the orthogonal projectors \(P_A:=\frac{1}{2}(I+J_A)\) and \(P_A^\bot=I-P_A\) satisfy 

\begin{equation}\label{createsgap} P_AHP_A\geq \alpha P_A,\quad P^\bot_AHP_A^\bot\leq -\alpha P^\bot_A \quad\text{for some }\alpha \in (0,1], \end{equation}
that is, \(\langle P_Ax,HP_Ax\rangle\geq \alpha \left\| P_Ax\right\|^2\quad\text{and}\quad \langle P_A^\bot x, HP_A^\bot x\rangle\leq -\alpha \|P_A^\bot x\|^2\) for all \(x\in \cH\).

\end{hypothesis}
Note that the condition \eqref{assumption1} in the hypothesis above is not always satisfied, see Example \ref{counterexample} below.

\begin{remark}\label{remarks} 
The following observations can be made under Hypothesis \ref{assumption1}.
\begin{enumerate}
\item  The involution \(J_A\) induces an orthogonal decomposition 
\[\cH=\ran(P_A) \oplus \ran(P_A^\bot)=:\cH_+\oplus \cH_-\] of the Hilbert space \(\cH\). Since \(P_A\) and \(P_A^\bot\) commute with \(A\), the subspaces \(\cH_+\) and \(\cH_-\) are reducing subspaces for the operator \(A\), see \cite[Section 2.5]{Wei} for this notion.   

With respect to this decomposition we have the block representations 
 \begin{equation*}J_A=\begin{pmatrix}I_{\cH_+}&0\\ 0&-I_{\cH_-}\end{pmatrix},\quad A=\begin{pmatrix}A_+&0\\ 0&A_-\end{pmatrix} \end{equation*} with  \(A_+:=P_AAP_A,\;A_-:=P_A^\bot AP_A^\bot\),\label{p:A+} where the operators \(A_{\pm}\)  are self-adjoint on the Hilbert spaces \(\cH_\pm\) with domains \(\dom(A)\cap \cH_\pm\), respectively. 

In this sense, we can always assume \(J_A\) and \(A\) to be diagonal block operator matrices of this structure. 
\item For any self-adjoint involution \(J_A\), a bounded self-adjoint operator \(\cH\) 
can always be represented as a block operator with respect to the decomposition induced by \(J_A\). Namely 
\begin{equation*}\qquad H=\begin{pmatrix}H_+&T\\ T^\ast &H_-\end{pmatrix}\quad \text{with}\;H_+:=P_AHP_A, \;H_-:=P_A^\bot H P_A^\bot,\; T:=P_AHP_A^\bot,\end{equation*} where \(H_\pm \colon \cH_\pm\to \cH_\pm\) are bounded, self-adjoint and \(T\colon \cH_-\to \cH_+\) is bounded.

Thus, condition \eqref{createsgap} can be rewritten as \begin{equation}\label{invert}H_+\geq \alpha I_{\cH_+}\quad\text{and}\quad H_-\leq  -\alpha I_{\cH_-}\text{ for some }\alpha \in (0,1].\end{equation} 
Note that any bounded self-adjoint operator \(H\) satisfying \eqref{invert} is automatically boundedly invertible, see \cite[Remark 2.8]{KMM}. 
\end{enumerate}
\end{remark}

Before we state the First Representation Theorem in the most general version considered here, recall that for a constant \(c\geq 0\) and a semibounded, self-adjoint operator \(S\geq -cI\), the domain identity 
\begin{equation}\label{dom}
\dom(|S|^{1/2})=\dom\big((S+(c+1)I)^{1/2}\big)\end{equation} 
holds by functional calculus. This allows to shift the considerations from a non-negative operator \(A^{1/2}\) to a strictly positive operator \((A+I)^{1/2}\) without changing the domain.  

\begin{theorem}[The First Representation Theorem]\label{1repgeneral}

\hspace{12pt}\\Assume Hypothesis \ref{assumption1}, and let \(\fb\) be the symmetric sesquilinear form given by
\begin{equation*}
	\fb[x,y]:=\langle A^{1/2}x, HA^{1/2}y\rangle,\quad \dom[\fb]=\dom(A^{1/2}).
\end{equation*}
Then, there exists a unique self-adjoint operator \(B\) with \(\dom(B)\subseteq \dom[\fb]\) and

\begin{equation*}
	\fb[x,y]=\langle x,By\rangle \quad \text{for all }x\in \dom[\fb], \;y\in \dom(B).
\end{equation*}
Moreover, the operator \(B\) is given by  
\begin{equation*}
	B=A^{1/2}HA^{1/2}
\end{equation*} on the natural domain 

\begin{equation*}
	\dom(B)=\{x\in \dom(A^{1/2})\;|\; HA^{1/2}x\in \dom(A^{1/2})\}.	
\end{equation*}
Furthermore, \(\dom(B)\) is a core for the operators \((A+I)^{1/2}\) and \(A^{1/2}\).
\begin{proof}

Consider the perturbed form \[\tilde{\fb}:=\tilde\fb+J_A\quad\text{on}\quad\dom[\tilde{\fb}]:=\dom[\fb]\] given by 
\[\tilde{\fb}[x,y]=\fb[x,y]+\langle x,J_Ay\rangle.\]
Using the domain equality \(\dom(A^{1/2})=\dom((A+I)^{1/2})\) and the commutativity of \(J_A\) and \(P_A\) with functions of \(A\), the perturbed form can be rewritten as
\begin{equation*}
	\tilde{\fb}[x,y]=\langle (A+I)^{1/2}x, \widetilde{H}(A+I)^{1/2}y\rangle, 
\end{equation*}
where 

\begin{equation}\label{def:tildeH}
	\begin{split}\widetilde{H}&:=(A^{1/2}(A+I)^{-1/2})HA^{1/2}(A+I)^{-1/2}+(A+I)^{-1}J_A\\ &=:H_0+(A+I)^{-1}J_A.\end{split}
\end{equation} 
 In this case, the operator \(\widetilde{H}\) is self-adjoint and bounded since \(A^{1/2}(A+I)^{-1/2}\) is bounded and self-adjoint by functional calculus.  

Let \(x\in \cH\), then we verify 
\[\langle P_Ax,H_0P_Ax\rangle=\langle  P_AA^{1/2}(A+I)^{-1/2}x,HP_AA^{1/2}(A+I)^{-1/2}x\rangle.\]
By hypothesis \eqref{createsgap}, we get the estimate
\[\langle P_Ax,H_0P_Ax\rangle\geq \alpha \langle P_AA^{1/2}(A+I)^{-1/2}x,P_AA^{1/2}(A+I)^{-1/2}x\rangle.\]
 Using the commutativity of \(P_A\) with functions of \(A\) and the equality \[(A^{1/2}(A+I)^{-1/2})^2=A(A+I)^{-1}=I-(A+I)^{-1},\] we can rewrite this estimate as 
\[\langle P_Ax,H_0P_Ax\rangle\geq	\alpha \langle P_Ax,P_Ax\rangle-\alpha\langle P_Ax,(A+I)^{-1}P_Ax\rangle.\]
With definition \eqref{def:tildeH}, the equality \(J_AP_A=P_A\), and \(\alpha\leq 1\), it follows that

\[P_A\widetilde{H}P_A\geq \alpha P_A-\alpha P_A(A+I)^{-1}P_A+P_A(A+I)^{-1}J_AP_A\geq\alpha P_A.\] 

In a similar way, noting that \(J_AP_A^\bot=-P_A^\bot\), we obtain that \(P^\bot_A\widetilde{H}P^\bot_A\leq -\alpha P^\bot_A\). As a consequence \(\widetilde{H}\) is boundedly invertible, see \cite[Remark 2.8]{KMM}.

Since \(A+I\) is strictly positive and \(\widetilde{H}\) bounded, boundedly invertible, we can apply the First Representation Theorem \cite[Theorem 2.3]{GKMV} to the form \(\tilde{\fb}\).
We obtain that the operator \[\widetilde{B}:=(A+I)^{1/2}\widetilde{H}(A+I)^{1/2}\] on its natural domain 
\[\dom(\widetilde{B})=\big\{x\in \dom(A+I)^{1/2}\;\big|\;\widetilde{H}(A+I)^{1/2}x\in \dom(A+I)^{1/2}\big\} \subseteq \dom(A^{1/2})\]  is the unique self-adjoint operator with \(\dom(\widetilde{B})\subseteq \dom[\tilde\fb]\) associated with the form \(\tilde{\fb}\), that is,

\[\tilde{\fb}[x,y]=\langle x, \widetilde{B}y\rangle \quad\text{for all}\;x\in \dom[\fb],\;y\in \dom(\widetilde{B}).\]
Additionally, \(\dom(\widetilde{B})\) is a core for \((A+I)^{1/2}\).

Setting \(B:=\widetilde{B}-J_A\) on \(\dom(B):=\dom(\widetilde{B})\), we obtain that 
\[\fb[x,y]=\langle x,By\rangle \quad\text{for all }x\in \dom[\fb],\; y\in \dom(B).\] 
Furthermore, we have
\[\begin{aligned}\dom(B)&=\big\{x\in \dom(A^{1/2})\;\big|\; A^{1/2}(A+I)^{-1/2}HA^{1/2}x\in \dom(A^{1/2})\big\}\\ &=
\big\{x\in \dom(A^{1/2})\;\big|\; (A+I)^{-1/2}HA^{1/2}x\in \dom(A)\big\}
\\&=\big\{x\in \dom(A^{1/2})\;\big|\; HA^{1/2}x\in \dom(A^{1/2})\big\}\end{aligned}\] and, hence,
\[B=(A+I)^{1/2}\widetilde{H}(A+I)^{1/2}-J_A=A^{1/2}HA^{1/2}.\]
The core property with respect to \(A^{1/2}\) is a direct consequence of the equivalence of the corresponding graph norms for \(A^{1/2}\) and \((A+I)^{1/2}\). 
\end{proof}
\end{theorem}
Note that the operator \(B\) constructed in the theorem above is in general not invertible although the perturbed operator \(\widetilde{B}=B+J_A\) has a bounded inverse.
The idea to create a spectral gap by a bounded perturbation is already present in \cite[Theorem 2.4]{Ve} by Veseli\'{c}. There however, it is not clear whether a suitable perturbation creating the gap exists and how it can be found. Also, a corresponding Second Representation Theorem is not considered in \cite{Ve}.

We now compare the two variants of the First Representation Theorem, Theorem \ref{1repgeneral} for non-negative operators and \cite[Theorem 2.3]{GKMV} for strictly positive operators, respectively.

\begin{remark}\label{noExtension}
Theorem \ref{1repgeneral} is a supplement to \cite[Theorem 2.3]{GKMV} in the sense that new pairs of operators \((A,H)\) can be treated, where \(A\) is allowed to be a non-negative operator. 

Theorem \ref{1repgeneral} is not an extension of \cite[Theorem 2.3]{GKMV} since there are cases of strictly positive operators \(A\) that cannot be covered by Theorem \ref{1repgeneral} but can be treated by \cite[Theorem 2.3]{GKMV}. A suitable  \(2\times 2\) matrix example for this is given by \[A:=\begin{pmatrix}2&0\\0&1/2\end{pmatrix},\quad H:=\begin{pmatrix}0&1\\1&0\end{pmatrix}.\] In this case, any self-adjoint involution \(J_A\) commuting with \(A\) is diagonal too.
So up to the choice of a sign, we would have  \(J_A=\pm\begin{pmatrix}1&0\\0&-1\end{pmatrix}\). Since \(P_AHP_A=0\) in this case the condition \eqref{createsgap} in Hypothesis \ref{assumption1} cannot be satisfied for the pair \((A,H)\).
\end{remark}

In the following, we provide an example where the operator \(B=A^{1/2}HA^{1/2}\) associated with the form \(\fb\) is only essentially self-adjoint if \(\min \sigma(A)=0\). Thus, condition \eqref{createsgap} ensures the closedness of the symmetric operator \(B=A^{1/2}HA^{1/2}\) if \(\min \sigma(A)=0\).
 
\begin{example}\label{counterexample}
Let \(\ell^{2,p}\) be the space of complex sequences \((a_k)_{k\in \mathbbm{N}}\), such that \[\sum_{k=1}^\infty k^p|a_k|^2<\infty.\] We abbreviate the underlying Hilbert space by \(\cH := \ell^2\oplus \ell^2\), where \(\ell^2:=\ell^{2,0}\).

On the Hilbert space \(\cH\), we define the self-adjoint operators \(A\) and \(H\) by \[A:=\bigoplus_{k\in \mathbbm{N}}\begin{pmatrix}k+1&0\\0& (k+1)^{-1}\end{pmatrix},\quad H:=\bigoplus_{k\in\mathbbm{N}}\begin{pmatrix}0&1\\1& 0\end{pmatrix}\] with \(\dom(H)=\ell^2\oplus \ell^2\) and 
\(\dom(A)=\ell^{2,2}\oplus \ell^{2}\subset \cH\). 
Here, the operators \(A\) and \(H\) are self-adjoint, \(\min \sigma(A)=0\) and \(H=H^{-1}\) is bounded.

Hypothesis \eqref{createsgap} is not satisfied since there is no suitable involution \(J_A\) commuting with \(A\). Indeed, assume that such a \(J_A\) exists, then, since \(A\) has a simple spectrum, \(J_A\) is a function of \(A\), see \cite[Proposition VIII.3.6]{DL}. By the diagonal block structure of \(A\), the operators \(J_A\) and \(P_A\) also must be block diagonal. 
Since each block of \(A\) itself is diagonal, the corresponding blocks of \(P_A\) are also diagonal.
Considering the block for \(k=1\), Remark \ref{noExtension} shows that it is not possible to find such a projector \(P_A\) since not even its first block can be constructed.

To see that \(A^{1/2}HA^{1/2}\) is not self-adjoint, let \(A_k\) and \(H_k\) denote the k-th block of \(A\) and \(H\), respectively. Then, we have 

\[A_k^{1/2}H_kA_k^{1/2}=H_k.\] In this sense, the symmetric operator  \(A^{1/2}HA^{1/2}\) is associated with the form \(\fb\) but is not closed on the natural domain \[\{x\in \dom(A^{1/2})\;|\;HA^{1/2}x\in \dom(A^{1/2})\}\subseteq\ell^{2,1}\oplus\ell^{2} \subset \ell^2\oplus \ell^2=\dom(H).\]
The closure of this operator is self-adjoint, so that the operator is only essentially self-adjoint.

\end{example}
The phenomenon appearing in the example above can be explained in the following way.
The operator \(A\) has arbitrarily large and arbitrarily small spectral parts. The operator \(H\) maps the large spectral parts to the small ones and vice versa in such a way that the product \(A^{1/2}HA^{1/2}\) remains bounded on its natural domain. 
The product is  not closed then. 

If \(A\) is strictly positive as in \cite{GKMV}, the large spectral parts have no counterpart to be mapped to. Therefore, the closedness of the product \(A^{1/2}HA^{1/2}\) on the natural domain is preserved. 

This distinguishes the case of strictly positive \(A\), where \(B\) is automatically self-adjoint, from the case of non-negative \(A\), where additional conditions have to be imposed. 

\subsection{The off-diagonal case}

In a similar way as in Theorem \ref{1repgeneral}, we consider forms \(\fb\) defined by an indefinite diagonal form with an off-diagonal additive form perturbation.

Let \(\cH\) be a Hilbert space and let \(\fa\) be a non-negative, closed sesquilinear form. By the First Representation Theorem for non-negative forms \cite[Theorem VI.2.1]{K}, the form \(\fa\) is associated with a non-negative self-adjoint operator \(A\). By \cite[Theorem VI.2.23]{K}, the form \(\fa\) is even represented by this operator. In this situation, we impose the following assumptions:

\begin{hypothesis}\label{assumption2}
Let \(J_A\) be a self-adjoint involution commuting with \(A\) and let 

\begin{equation*}\cH=\cH_+\oplus \cH_-,\quad\cH_\pm:=\ran(I\pm J_A)\end{equation*} be the orthogonal decomposition induced by \(J_A\). Suppose that \(\fv\) is a symmetric sesquilinear form on \[\dom[\fv]\supseteq\dom[\fa]=\dom(A^{1/2}),\]and assume that \(\fv\) is \((\fa+I)\)-bounded, which means that there exists a finite constant \(\beta\) with
\begin{equation*}\label{eq:offdiag1}
	|\fv[x]|\leq \beta\left\|(A+I)^{1/2}x\right\|^2=\beta(\fa+I)[x],\; x\in \dom(A^{1/2}).
\end{equation*} 
Suppose furthermore that \(\fv\) is off-diagonal with respect to the decomposition induced by \(J_A\), that is, 

\begin{equation*}\label{eq:offdiag2}
	\fv[J_Ax,y]=-\fv[x,J_Ay]\quad\text{for all }x,y\in \dom[\fa].\end{equation*}
\end{hypothesis}
The forms \(\fv\) satisfying  Hypothesis \ref{assumption2} have a special structure with respect to the operator \((A+I)^{1/2}\).

\begin{remark}\label{operatorV}
Let \(\fv\) be a form satisfying Hypothesis \ref{assumption2}. Then \(\fv\) can  explicitly be rewritten on \(\dom[\fa]\) as  

\begin{equation*}
\fv[x,y]=\langle S(A+I)^{1/2}x, (A+I)^{1/2}y\rangle,\quad x,\;y\in \dom[\fa]=\dom(A^{1/2}),
\end{equation*}
where \(S\) is a bounded operator with \(\left\|S\right\|\leq \beta\), see \cite[Lemma VI.3.1]{K}. 
Since the form \(\fv\) is  off-diagonal with respect to the decomposition induced by \(J_A\), we have that \[S=\begin{pmatrix}0&T\\ T^\ast&0\end{pmatrix},\; T:=P_ASP_A^\bot\colon \cH_-\to \cH_+\] is off-diagonal with respect to \(\cH=\cH_+\oplus \cH_-\).    
\end{remark}
We are now ready to formulate the First Representation Theorem in this off-diagonal setting extending \cite[Theorem 2.5]{GKMV} to the case of \(\min\sigma(A)=0\).

\begin{theorem}[The First Representation Theorem in the off-diagonal case] \label{1repoffdiag}

\hspace{1pt}

Assume Hypothesis \ref{assumption2} and let \(\fb\) be the symmetric sesquilinear form on\\ \(\dom[\fb]=\dom[\fa]\) given by

\begin{equation*}
	\fb[x,y]:=\fa[x,J_Ay]+\fv[x,y].
\end{equation*}
Then, there exists a unique self-adjoint operator \(B\) on \(\dom(B)\subseteq \dom[\fb]\) such that 
\begin{equation*}
	\fb[x,y]=\langle x,By\rangle \quad\text{for all}\;x\in \dom[\fa],\; y\in \dom(B).
\end{equation*}
Furthermore \(\dom(B)\) is a \emph{form core} for \(\fa\), that is, \(\dom(B)\) is dense in \(\dom[\fa]\) with respect to the norm \(\sqrt{(\fa+I)[\;\cdot\;]}\).

\begin{proof}

Consider the perturbed form \(\tilde{\fb}\) on \(\dom[\tilde\fb]=\dom[\fb]\) given by 
\[\tilde{\fb}[x,y]:=\fb[x,y]+\langle x,J_Ay\rangle=(\fa+I)[x,J_Ay]+\fv[x,y],\]
and let \(\fh\) be the bounded form given by \[\fh[x,y]:=\tilde{\fb}[(A+I)^{-1/2}x,(A+I)^{-1/2}y].\] Then, by Remark \ref{operatorV}, the form \(\fh\) corresponds to the bounded block operator \begin{equation}\label{tildeH}\widetilde{H}:=J_A+S=\begin{pmatrix}I_{\cH_+}&T\\T^*&-I_{\cH_-}\end{pmatrix},\end{equation} where \(T=P_ASP_A^\bot\colon \cH_-\to \cH_+\). By \cite[Remark 2.8]{KMM} the operator \(\widetilde{H}\) is boundedly invertible.
We compute that
\begin{equation}\label{tildeB}\tilde{\fb}[x,y]=\fh[(A+I)^{1/2}x,(A+I)^{1/2}y]=\langle (A+I)^{1/2}x,\widetilde{H}(A+I)^{1/2}y\rangle,\;x,y\in \dom[\tilde\fb],\end{equation}
so that by \cite[Theorem 2.3]{GKMV} the form \(\tilde\fb\) is associated with a unique self-adjoint operator \(\widetilde{B}\) satisfying \(\dom(\widetilde{B})\subseteq \dom[\fb]\). Thus, we have that
\[\tilde{\fb}[x,y]=\langle x,\widetilde{B}y\rangle,\quad\text{for all}\;x\in \dom[\fa], \;y\in \dom(\widetilde{B}).\]
The self-adjoint operator associated with the form \(\fb\) is then  \(B:=\widetilde{B}-J_A\). The core property is a direct consequence of the corresponding core property in Theorem \ref{1repgeneral}.
\end{proof}
\end{theorem}

\begin{remark}\label{def:B+JA}
The operators \(\widetilde{H}\)\label{p:H} and \(\widetilde{B}=B+J_A\)\label{p:B} associated with \(\tilde\fb=\fb+J_A\), appearing in the proofs of the First Representation Theorem in the general case (Theorem \ref{1repgeneral}) and in the off-diagonal case (Theorem \ref{1repoffdiag}) are boundedly invertible, see \cite[Remark 2.8]{KMM} and \cite[Theorem 2.3]{GKMV}. 
More concretely, we even have the estimate \((-c,c)\subset  \rho(B+J_A)\), where \(c:=\|\widetilde{H}^{-1}\|^{-1}\) is a lower estimate on the spectral gap of \(\widetilde{H}\). In the general case, we have \(c\leq \alpha\leq 1\) and \(c\leq 1\) in the off-diagonal case.

Additionally, \(\widetilde{H}\) is bounded  and \(\widetilde{B}\) can be represented as
\[\widetilde{B}=(A+I)^{1/2}\widetilde{H}(A+I)^{1/2}.\] \end{remark}

In contrast to the general case, Theorem \ref{1repoffdiag} is an extension to the corresponding Theorem \cite[Theorem 2.5]{GKMV} (cf.~ Remark \ref{noExtension} for the general case). This is contained in the following remark.
 
\begin{remark}\label{equivbounds}
Theorem \ref{1repoffdiag} is a generalisation of \cite[Theorem 2.5]{GKMV}. 

Indeed, for strictly positive \(\fa\geq c>0\), the two sided estimate 
\[\fa[x]\leq (\fa+I)[x]=\fa[x]+c^{-1}c\left\|x\right\|^2\leq \fa[x]+c^{-1}\fa[x]=(1+c^{-1})\fa[x]\]
implies the equivalence between \(\fa\)-boundedness and \((\fa+I)\)-boundedness in this case.
This yields that forms satisfying the requirements of Theorem \ref{1repoffdiag} also satisfy those of \cite[Theorem 2.5]{GKMV}. 
\end{remark}

Note that the First Representation Theorem in the off-diagonal case is already contained in \cite{Ne} by Nenciu. That work however does not consider the Second Representation Theorem.

\begin{remark}\label{Nenciu}
The First Representation Theorem \ref{1repoffdiag} we give here is a special case of \cite[Theorem 2.1]{Ne}. To see this, consider \(U:=J_A\) as the unitary part of the polar decomposition of the self-adjoint operator \(J_AA\). We then set \[\fh_{A_1}[x,y]:=\langle |J_AA|^{1/2}x,J_A|J_AA|^{1/2}y\rangle+1\langle x, J_Ay \rangle\] in equation (2.5) of \cite{Ne}. 
The off-diagonal form \(\fv\) is, by \cite[Definition 2.1]{Ne}, then a form perturbation of \(J_AA\). Indeed, the first two conditions in \cite[Definition 2.1]{Ne} can be seen directly, namely \[\dom[\fv]\supseteq\dom(A^{1/2}) =\dom(|J_AA|^{1/2})\] and \[|\fv[x,y]|\leq \beta \|(A+I)^{1/2}x\|\cdot \|(A+I)^{1/2}y\|.\] It remains to note that the operator \(\widetilde{H}=J_A+S\) is boundedly invertible by the assumption that \(\fv\) is off-diagonal. Since we can translate \(J_A\equiv T,\;S\equiv V_1\) into the notation of \cite{Ne}, the sum \(T+V_1\) has a bounded inverse and thus also the last condition in \cite[Theorem 2.1]{Ne} is satisfied.      
\end{remark}

We now compare the general case in Theorems \ref{1repgeneral} to the off-diagonal case in Theorem \ref{1repoffdiag}.
\begin{remark}
\begin{enumerate}
\item
The off-diagonal case in Theorem \ref{1repoffdiag} is a special case of Theorem \ref{1repgeneral}, where the Hypothesis \eqref{createsgap} is automatically satisfied. In this case, the involution \(J_A\) creating the spectral gap is already given by the diagonal structure of the form \(\fa\).
Indeed, in this situation, equations \eqref{tildeB} and \eqref{tildeH} imply that the perturbed form \(\fb+J_A\) and thus also \(\fb\) satisfy the First Representation Theorem \ref{1repoffdiag}. In the general case however, finding a suitable perturbation \(J_A\) may be difficult or not possible at all as Example \ref{counterexample} illustrates.
\item 
The difference between the results of Theorems \ref{1repgeneral} and \ref{1repoffdiag} lies in the representation of the operator \(B\) associated to the form \(\fb\). In the general case, we have the product formula \[B=A^{1/2}HA^{1/2}\] involving only the operators \(A\) and \(H\) defining the form \(\fb\). 
As a consequence, an explicit representation \[\Ker(B)=\{x\in \dom(A^{1/2})\;|\; HA^{1/2}x\in \Ker(A^{1/2})\}\] can be directly deduced.

In the off-diagonal case, if the form \(\fv\) is only \((\fa+I)\)-bounded but not bounded with respect to the form \(\fa\), a corresponding operator \(H\) seems to be artificial.

The best representation we have in this situation is 
 \[B=(A+I)^{1/2}\widetilde{H}(A+I)^{1/2}-J_A.\]

Indeed, the operator \(B\) cannot be written as a product with respect to \(A^{1/2}\) and a bounded operator \(H\) like in the first case, since in this case the operator \(H\) would formally be given by the block operator matrix
\[\begin{pmatrix}I_{\cH_+}& A_+^{-\frac{1}{2}}(A_++I)^{\frac{1}{2}}T(A_-+I)^{\frac{1}{2}}A_-^{-\frac{1}{2}}\\A_-^{-\frac{1}{2}}(A_-+I)^{\frac{1}{2}}T^\ast(A_++I)^{\frac{1}{2}}A_+^{-\frac{1}{2}}& -I_{\cH_-}\end{pmatrix}.\] 
If \(\min \sigma(A_\pm)=0\), the off-diagonal entries of this matrix are either unbounded or may not exist at all if \(A_\pm\) has a non-trivial kernel. So if such an operator \(H\) exists, it would in general be unbounded. In this case, an explicit representation of the kernel is more difficult to obtain. This will be carried out in Theorem \ref{thm:incl} below.

For a strictly positive form \(\fa\), respectively operator \(A\), however, the off-diagonal part \(\fv\) is even \(\fa\)-bounded by Remark \ref{equivbounds}. Thus \(B=A^{1/2}HA^{1/2}\) is still valid by direct application of \cite[Theorem 2.5 and Lemma 2.2]{GKMV} with the operator \[H=\begin{pmatrix}I_{\cH_+}&T\\T^\ast&-I_{\cH_-}\end{pmatrix}.\] 
In this case, the kernel can be represented as in the general case,
\[\Ker(B)=\{x\in \dom(A^{1/2})\;|\; HA^{1/2}x\in \Ker(A^{1/2})\}.\]
\end{enumerate}
\end{remark}

\subsection{Representation of the kernel in the off-diagonal case}

For strictly positive \(\fa\), the operator \(B\) in \cite[Theorem 2.5]{GKMV} associated with the form \(\fb\) defined by \[\fb[x,y]:=\fa[x,J_Ay]+\fv[x,y]\]  is boundedly invertible. In the situation where \(A\) is only non-negative, the operator \(B\) may have a non-trivial kernel. 
In the following, we give a description for the kernel in the off-diagonal case of Theorem \ref{1repoffdiag}. In the authors Ph. D. thesis \cite{ste}, this description is used explicitly in the case of the Stokes operator on unbounded domains.

Recall that the operators  \(J_A\) and \(P_A\) commute with \(A\) by assumption and the decomposition \(\cH=\cH_+\oplus \cH_-\) reduces \(A\) (see \cite[Satz 2.60]{Wei}), so that with respect to this decomposition, one has \(A=A_+\oplus A_-\) on \(\dom(A)=\dom(A_+)\oplus \dom(A_-)\) with self-adjoint operators \(A_\pm\).  

\begin{definition}\label{defker} 
Assume Hypothesis \ref{assumption2}. 
We set for brevity \begin{equation}\label{p:Lpm}
\fL_\pm := \big\{x_\pm\in\dom(A_\pm^{1/2})\,|\,\fv[x_+\oplus 0,0\oplus x_-]=0\quad \text{for all}\; x_\mp\in\dom(A_\mp^{1/2})\big\}.
\end{equation}
\end{definition}

Note that \(\fL_+\oplus \{0\}\) and \(\{0\}\oplus\fL_-\) are not necessarily subsets of \(\dom(B)\) or closed. This fact has to be taken into account for the computation of the kernel of the associated operator \(B\).

We are now ready to give a representation for the kernel of \(B\) with respect to the kernels of the components \(A_\pm\).
This is a generalisation of \cite[Theorem 2.2]{KMM} to the case of unbounded operators respectively forms.
\begin{theorem}\label{thm:incl}
Let \(B\) be the operator associated with the form \(\fb\) in  Theorem \ref{1repoffdiag}. 
Then we have that
\begin{equation*}\Ker(B) = \left(\Ker(A_+) \cap \fL_+\right) \oplus \left( \Ker(A_-) \cap \fL_-\right).
\end{equation*}

\begin{proof}
Suppose first that \(x=x_+\oplus x_-\in\Ker(B)\subseteq \dom(A^{1/2})\) with respect to \(\cH=\cH_+\oplus \cH_-\). By the First Representation Theorem \ref{1repoffdiag}, it follows that
 \[0=\langle y, Bx\rangle=\fb[y,x]=\fa[y,J_Ax]+\fv[y,x]\text{ for all }y\in\dom(A^{1/2}).\] 
Writing \(y=y_+\oplus y_-\) with \(y_\pm\in \dom(A^{1/2}_\pm)\), the Second Representation Theorem for the non-negative form \(\fa\) (see \cite[Theorem VI.2.23]{K}) yields that
\begin{equation*}
\langle A_+^{1/2}y_+, A_+^{1/2}x_+\rangle_{\cH_+} - \langle A_-^{1/2}y_-, A_-^{1/2}x_-\rangle_{\cH_-} +\fv[y_+\oplus 0,0\oplus x_-]+\fv[0\oplus y_-,x_+\oplus 0]=0.
\end{equation*}
Choosing \(y_-=0\), respectively \(y_+=0\), we arrive at
\begin{equation}\label{eq:2}
\langle A_+^{1/2}y_+, A_+^{1/2}x_+\rangle_{\cH_+} +\fv[y_+\oplus 0,0\oplus x_-]=0
\end{equation}
and
\begin{equation}\label{eq:3}
- \langle A_-^{1/2}y_-, A_-^{1/2}x_-\rangle_{\cH_-} +\fv[0\oplus y_-,x_+\oplus 0]=0,
\end{equation}
respectively. In particular, if \(y_+=x_+\) and \(y_-=x_-\), we have that
\begin{equation}\label{eq:4}
\big\| A_+^{1/2}x_+\big\|_{\cH_+}^2 +\fv[x_+\oplus 0,0\oplus x_-]=0
\end{equation}
and
\begin{equation}\label{eq:5}
- \big\| A_-^{1/2}x_-\big\|_{\cH_-}^2 +\overline{\fv[x_+\oplus 0,0\oplus x_-]}=0,
\end{equation}
Suppose that $x_+\notin\Ker(A_+)=\Ker(A_+^{1/2})$. Then, from \eqref{eq:4} we get that \[\fv[x_+\oplus 0,0\oplus x_-]<0\] and from \eqref{eq:5} follows that \[\fv[x_+\oplus 0 ,0\oplus x_-]=\overline{\fv[x_+\oplus 0,0\oplus x_-]}\geq 0,\] which yields a contradiction. Thus, \(x_+\in\Ker(A_+)\). 

Using  equation \eqref{eq:2} again, we obtain that \[\fv[y_+\oplus 0,0\oplus x_-]=0\quad\text{ for all }y_+\in\dom(A_+^{1/2})\] and, hence, \(x_-\in\fL_-\).
Similarly, one proves that \(x_-\in\Ker(A_-)\) and \(x_+\in\fL_+\). This proves the inclusion 
\[\Ker(B) \subseteq \left(\Ker(A_+) \cap \fL_+\right) \oplus \left(\Ker(A_-) \cap \fL_-\right).\]

We now turn to the converse inclusion.
By the \((\fa+I)\)-boundedness of \(\fv\) in Hypothesis \ref{assumption2}, the auxiliary form
\begin{equation*}
\ft[x_+,y_-]:= \fv[(A_++I_{\cH_+})^{-1/2}x_+\oplus 0,0\oplus (A_-+I_{\cH_-})^{-1/2}y_-],\quad x_+\in \cH_+,\;y_-\in\cH_-
\end{equation*}
is bounded.
Hence, there exists a bounded operator \(T:\cH_-\to\cH_+\) such that
\begin{equation*}
\ft[x_+,y_-]=\langle x_+, Ty_-\rangle_{\cH_+}.
\end{equation*}
Noticing \(\dom(A_\pm^{1/2})=\ran\big((A_\pm+I_{\cH_\pm})^{-1/2}\big)\), we get that
\begin{equation}\label{L+}\fL_+=\{(A_++I_{\cH_+})^{-1/2}x\, |\, x\in\Ker(T^\ast) \} = (A_++I_{\cH_+})^{-1/2} \Ker(T^\ast),
\end{equation}
In the same way we obtain that \begin{equation}\label{L-}\fL_-= (A_-+I_{\cH_-})^{-1/2} \Ker(T).\end{equation}
Let 
$x_+\in \Ker(A_+) \cap \fL_+$ and $x_-\in \Ker(A_-) \cap \fL_-$. Then, by \eqref{L+} and \eqref{L-}, there exist $u_+\in\Ker(T^\ast)$ and $u_-\in\Ker(T)$
such that 
\begin{equation*}
x_+=(A_++I_{\cH_+})^{-1/2} u_+\quad\text{and}\quad x_-=(A_++I_{\cH_+})^{-1/2} u_-.
\end{equation*}
Obviously, from $x_+\in\Ker(A_+)\subset\dom(A_+)$, it follows that $u_+\in\dom(A_+^{1/2})$. 

Similarly, we have $u_-\in\dom(A_-^{1/2})$.

We claim that $u_+\in\Ker(A_+)$ and $u_-\in\Ker(A_-)$. Indeed, we have that 
\begin{equation*}
x_+ = (A_++I_{\cH_+})x_+ = (A_++I_{\cH_+})^{1/2} u_+,
\end{equation*}
which implies that $u_+=(A_++I_{\cH_+})^{-1/2}x_+$ and, thus, $u_+\in\dom(A_+^{3/2})$. Hence, we arrive at the conclusion that
\begin{equation*}
A_+ u_+= (A_++I_{\cH_+})^{-1/2}A_+x_+=0,
\end{equation*}
which proves that $u_+\in\Ker(A_+)$. In the same way we also have $u_-\in\Ker(A_-)$.
By Remark \ref{def:B+JA} and equation \eqref{tildeH} we get the following representation:
\begin{equation}\label{eq:B}
B=(A+I)^{1/2} \widehat{H}(A+I)^{1/2}
\end{equation}
with the operator 
\begin{equation}\label{B:repr}
\widehat{H} :=\begin{pmatrix} I_{\cH_+}-(A_++I_{\cH_+})^{-1} & T \\ T^\ast & -I_{\cH_-}+(A_-+I_{\cH_-})^{-1}\end{pmatrix}.
\end{equation}
This representation  of \(B\) follows from 
\begin{equation*}
B=(A+I)^{1/2} \widetilde{H}(A+I)^{1/2}-J_A
\end{equation*}
together with
\begin{equation*}
\widetilde{H} = \begin{pmatrix} I_{\cH_+} & T \\ T^\ast & -I_{\cH_-}\end{pmatrix}.
\end{equation*} Identifying \(x=x_+\oplus x_-\) with the vector \(\begin{pmatrix}x_+\\x_-\end{pmatrix}\), we compute
\begin{equation*}
\begin{split}&
\widehat{H}(A+I)^{1/2}\begin{pmatrix} x_+ \\ x_-\end{pmatrix}\\ &= \begin{pmatrix} (A_++I_{\cH_+})^{1/2} x_+ - (A_++I_{\cH_-})^{-1/2} x_+ + T(A_-+I_{\cH_-})^{1/2} x_- \\ T^\ast(A_++I_{\cH_+})^{1/2} x_+ - (A_-+I_{\cH_-})^{1/2} x_- + (A_-+I_{\cH_-})^{-1/2} x_- \end{pmatrix}\\
&= \begin{pmatrix} u_+ - (A_++I_{\cH_+})^{-1}u_+ + T u_- \\ T^\ast u_+ -u_- + (A_-+I_{\cH_-})^{-1}u_- \end{pmatrix}\\ &=\begin{pmatrix}A_+(A_++I_{\cH_+})^{-1}u_+\\ -A_-(A_-+I_{\cH_-})^{-1}u_-\end{pmatrix}  =\begin{pmatrix}(A_++I_{\cH_+})^{-1}A_+u_+\\ -(A_-+I_{\cH_-})^{-1}A_-u_-\end{pmatrix}=0.
\end{split}
\end{equation*}
{}From  the representation \eqref{B:repr}, it follows that \(x\in \Ker(B)\) which completes the proof.
\end{proof}
\end{theorem}

%%%%%%%%%%%%%%%%%%%%%%%%%%%%%%%%%%%%%%%%%%%%%%%%%%%%%%
\section{The Second Representation Theorem}\label{sec2rep}
%%%%%%%%%%%%%%%%%%%%%%%%%%%%%%%%%%%%%%%%%%%%%%%%%%%%%
In this section, we consider the Second Representation Theorem simultaneously in the situations where either Hypothesis \ref{assumption1} or \ref{assumption2} is satisfied.

These situations can be treated simultaneously since the operator \(B\) associated with the form \(\fb\) can be represented in the same way, see Remark \ref{def:B+JA}. Namely  
\[B=(A+I)^{1/2}\widetilde{H}(A+I)^{1/2}-J_A\] 
holds in both situations. We define the sign function by choosing \(\sign(0):=0\).

\begin{theorem}[The Second Representation Theorem]\label{2rep}

Let \(\fb\) be given as in Theorem \ref{1repgeneral} or Theorem \ref{1repoffdiag}, and  let \(B\) be the associated operator. Furthermore, suppose that 

\begin{equation}\label{eq:dB=dA}
\dom(|B|^{1/2})=\dom(A^{1/2}).
\end{equation}
Then, the operator \(B\) \emph{represents} the form \(\fb\), that is, 

\begin{equation}\fb[x,y]=\langle |B|^{1/2}x,\sign(B)|B|^{1/2}y\rangle \quad \text{for all}\; x,y \in \dom[\fb]=\dom(|B|^{1/2})\end{equation} 
holds.

\end{theorem}

Note that this theorem gives the correspondence between the form \(\fb\) and the operator \(B\) under  suitable assumptions. 
However, it is not clear whether Hypothesis \ref{assumption1} already implies condition \eqref{eq:dB=dA}. 
For strictly positive \(A\), an example where \eqref{eq:dB=dA} is not satisfied is given by \cite[Example 2.11]{GKMV}. 
This example however does not satisfy condition \eqref{createsgap} of Hypothesis \ref{assumption1}.

Before we turn to the proof, we need some preparations starting with the well known Heinz Inequality in the formulation of  \cite[Lemma 3.2.3]{S}.
\begin{lemma}[The Heinz inequality]\label{Heinz}

Let \(\cH_1,\cH_2\) be two Hilbert spaces and let \(S\colon \cH_1\to \cH_2\) be a bounded linear operator.
Assume that \(T_1\) and \(T_2\) are self-adjoint injective operators on \(\cH_1\) and \(\cH_2\), respectively. Suppose that \(S\) maps \(\dom(T_1)\) into  \(\dom(T_2)\) and that there is a finite constant \(c\) such that \begin{equation}\label{estim}||T_2Sx||_{\cH_2}\leq c\cdot||T_1x||_{\cH_1}\quad\text{for all}\;x\in \dom(T_1).\end{equation}
Then \(S\) maps \(\dom(T_1^\nu)\) into \(\dom(T_2^\nu)\) for all \(0 \leq \nu\leq 1\).  
\end{lemma}
As a direct consequence of the Heinz Inequality, we get the following corollary.
\begin{corollary}\label{Heinz2}
Let \(T_1,T_2\) be two strictly positive, self-adjoint operators in the Hilbert space \(\cH\). If the domain equality 
\[\dom(T_1)=\dom(T_2)\] holds, then also the domain equality for the roots holds, that is, 
\[\dom(T_1^\nu)=\dom(T_2^\nu)\quad\text{for all }\nu \in [0,1].\] 
\end{corollary}

\begin{lemma}\label{Heinzapplied}

Let \(B\) and \(J_A\) be the operators in either version of the First Representation Theorem \ref{1repgeneral} or \ref{1repoffdiag}. Then the domain equality
\begin{equation}\label{eq:domB}\dom(|B|^{1/2})=\dom(|B+J_A|^{1/2})\end{equation} holds. 

\begin{proof}
In both situations, we have by Remark \ref{def:B+JA} that \(B+J_A\) is boundedly invertible. The operator \(|B|+I\) is boundedly invertible by functional calculus. Clearly, one has \[\dom(|B+J_A|)=\dom(B+J_A)=\dom(B)=\dom(|B|)=\dom(|B|+I).\]
By Corollary \ref{Heinz2} and equation \eqref{dom}, we have the domain equality
\[\dom(|B+J_A|^{1/2})=\dom((|B|+I)^{1/2})=\dom(|B|^{1/2}).\qedhere\] 
\end{proof}
\end{lemma}
We now turn to the proof of the Second Representation Theorem \ref{2rep}.

\begin{proof}[Proof of Theorem \ref{2rep}]
Clearly, by equation \eqref{dom}, we have \[\dom(A^{1/2})=\dom((A+I)^{1/2})\] and, by Lemma \ref{Heinzapplied} \[\dom(|B|^{1/2})=\dom(|B+J_A|^{1/2}).\] 

Taking into account \eqref{eq:dB=dA}, the First Representation Theorem \ref{1repgeneral}, respectively \ref{1repoffdiag}, yields that

\begin{equation*}\fb[x,y]=\langle |B|^{1/2}x, \sign(B)|B|^{1/2}y\rangle\; \text{ for all } x\in\dom(|B|^{1/2}),\;y\in \dom(B).\end{equation*}
Fix \(x\in\dom(|B|^{1/2})\) and define the functionals \(l_1\) and \(l_2\) on \(\dom(A^{1/2})\) by \[l_1(y):=\fb[x,y]=\langle A^{1/2}x,HA^{1/2}y\rangle ,\quad l_2(y):=\langle |B|^{1/2}x,\sign(B)|B|^{1/2}y\rangle.\] These two functionals agree on \(\dom(B)\) and we show that they still agree on the whole of \(\dom(A^{1/2})\).

To do this, we proof that the shifted functionals  \[\tilde{l}_1(y):=l_1(y)+\langle x,J_Ay\rangle ,\quad\tilde{l}_2(y):= l_2(y)+\langle x,J_Ay\rangle\] agree on \(\dom((A+I)^{1/2})=\dom(|B+J_A|^{1/2})\), then also the original functionals \(l_1\) and \(l_2\) agree. 

For the functional \(\tilde{l}_1\), we get the representation  
\[\tilde{l}_1(y)=\langle (A+I)^{1/2}x,\widetilde{H}(A+I)^{1/2}y\rangle\]
with the bounded and boundedly invertible operator \(\widetilde{H}\), see Remark \ref{def:B+JA}.

In a similar way, we get  \[\tilde{l}_2(y)=\langle |B+J_A|^{1/2}x,G|B+J_A|^{1/2}y\rangle\] 
with the bounded, boundedly invertible operator \[G:=\sign(B+J_A).\]
By the boundedness of \(\widetilde{H}\), we have that the functional \(\tilde{l}_1\) is continuous on the Hilbert space \[\big(\dom((A+I)^{1/2}), \langle (A+1)^{1/2}\;\cdot\;,(A+1)^{1/2}\;\cdot\;\rangle\big)=:\cH_{A+I}.\]

A similar argument shows that \(\tilde{l}_2\) is continuous on the Hilbert space  
\[\big(\dom(|B+J_A|^{1/2}), \langle |B+J_A|^{1/2}\;\cdot\;,|B+J_A|^{1/2}\;\cdot\;\rangle\big)=:\cH_{B+J_A}.\]
The domain equality \(\dom((A+I)^{1/2})=\dom(|B+J_A|^{1/2})\) implies that the operator \((A+I)^{1/2}\) is \(|B+J_A|^{1/2}\)-bounded and vice versa. 
Since both operators \(A+I\) and \(|B+J_A|\) are strictly positive, the operators \[(A+I)^{1/2}|B+J_A|^{-1/2}\quad\text{and}\quad|B+J_A|^{1/2}(A+I)^{-1/2}\] are positive and bounded. Consequently, the norms 
\[|x|_{A+I}:=\big\|(A+I)^{1/2}x\big\|\;\text{ and }|x|_{B+J_A}:=\big\| |B+J_A|^{1/2}x\big\|\] are equivalent on the Hilbert space \(\cH_{A+I}\).

Since \(\dom(B)=\dom(B+J_A)=\dom(|B+J_A|)\) is a core for the operator \(|B+J_A|^{1/2}\) (see \cite[Theorem V.3.35]{K}), it follows that \(\dom(B)\) is dense in \(\cH_{A+I}\). The two functionals \(\tilde{l}_1,\tilde{l}_2\) are both closed since both \(\widetilde{H}\) and \(G\) are boundedly invertible.
By the uniqueness of the closure, we have that \(\tilde{l}_1=\tilde{l}_2\text{ on }\cH_{A+I}\) and the claim follows. 
\end{proof}
 
\subsection{Domain stability condition}\label{secDom}

The domain stability condition \[\dom(|B|^{1/2})=\dom(A^{1/2})\] in  the hypothesis  of Theorem \ref{2rep} is in general hard to verify directly. 
We are thus interested in equivalent characterisations and in sufficient criteria for this condition. 
Below, equivalent charcterisations are given in Theorem \ref{equivalent} and sufficient criteria are contained in Lemma \ref{sufficient}.
The conditions and criteria are natural extensions to the ones presented in \cite{GKMV} for strictly positive operators \(A\).

In order to start the investigation of the stability condition, we need the following tools.
The first one is the Second Resolvent Identity (see, e.g. \cite[Section 2.2]{Sch}).

\begin{lemma}[The Second Resolvent Identity]

Let \(T_1,T_2\) be closed linear operators defined on the same domain \(\dom(T_1)=\dom(T_2)\). Assume that the resolvent sets \(\rho(T_1)\) and \(\rho(T_2)\) intersect. Let \(R_\lambda(T_i)=(\lambda I-T_i)^{-1}\) be the resolvent of \(T_i\).
Then, for any \(\lambda \in\rho(T_1)\cap\rho(T_2)\), the difference of the resolvents satisfies
\begin{equation}\label{eq:resolvent}
R_\lambda(T_1)-R_\lambda(T_2)=R_\lambda(T_1)(T_1-T_2)R_\lambda(T_2)=R_\lambda(T_2)(T_1-T_2)R_\lambda(T_1).
\end{equation}
\end{lemma}

Another tool we use is the following: 
\begin{lemma}[{\cite[Lemma 3.1]{GKMV}}]\label{comparison}

Let \(\left(\cH,\langle \;\cdot\;,\;\cdot\;\rangle\right)\) and \(\left(\cH',\langle \;\cdot\;,\;\cdot\;\rangle'\right)\) be Hilbert spaces. Assume that \(\cH'\) is continuously imbedded in \(\cH\). 

If \(C\colon \cH \to \cH\) is a bounded map leaving the set \(\cH'\) invariant, then the operator \(C'\) induced by \(C\) on \(\cH'\) is bounded in the topology of \(\cH'\). 
\end{lemma}

In the following investigations, we want to consider the sign of the operator \(B\) as a unitary operator. 
Since \(B\) may have a kernel, we need to choose the sign of zero to be either \(+1\) or \(-1\). All the following statements are independent of this concrete choice, so we leave this choice open. However, in Lemma \ref{sufficient} below, it is convenient to have this freedom of choice.
We define the unitary version of the sign by
\begin{equation}\label{p:sgn}\sgn(x):=\begin{cases} -1,&  x<0,\\ s,& x=0,\\ +1,& x>0,\end{cases} \end{equation} for some \(s\in \{-1,1\}\). 
Note that, by functional calculus, \(\sign(B)f(B)=\sgn(B)f(B)\) for any function \(f\) satisfying \(f(0)=0\).
Furthermore, since the interval \((-1,1)\) is not contained in the range of the function \(f\) defined by \(f(x):=x+\sgn(x)\), the operator \(B+\sgn(B)\) is boundedly invertible by functional calculus.

We now need the following observations. 

\begin{lemma}\label{tildes}
Let the assumptions of the First Representation Theorem \ref{1repgeneral} or \ref{1repoffdiag} be satisfied.

Then, the indefinite operators  
\begin{equation}\label{tildeX}(A+I)^{1/2}(B+\sgn(B))^{-1}(A+I)^{1/2}\text{ defined on }\dom(A^{1/2})\end{equation} and \begin{equation}\label{tildeY}(A+I)^{-1/2}(B+\sgn(B))(A+I)^{-1/2}\;\text{ defined on the dense set }(A+I)^{1/2}\dom(B)\end{equation} 
can be extended, by closure, to bounded operators on \(\cH\). Since they are inverse to each other, they are boundedly invertible.
\begin{proof}
The operator in \eqref{tildeX} is obviously  densely defined. {}From \(\dom(B)=\dom(B+J_A)\) and Remark \ref{def:B+JA}, it follows that
\[(A+I)^{1/2}\dom(B)=\widetilde{H}^{-1}\dom(A^{1/2}),\] so that the operator in \eqref{tildeY} also is densely defined.
Since \(B+J_A=(A+I)^{1/2}\widetilde{H}(A+I)^{1/2}\) is boundedly invertible, we have that \[L:= (A+I)^{1/2}(B+J_A)^{-1}(A+I)^{1/2}\] is bounded.

Since both operators \(B+\sgn(B)\) and \(B+J_A\)  are closed, boundedly invertible and defined on \(\dom(B)\), we have that \(0\in \rho(B+\sgn(B))\cap\rho(B+J_A)\). We now apply the Second Resolvent Identity \eqref{eq:resolvent} in both variants. Setting for brevity \(J:=\sgn(B)\) and \(S:=J-J_A\), we obtain that

\[\begin{aligned}(B+\sgn(B))^{-1}&=(B+J_A)^{-1}+(B+J_A)^{-1}S(B+J)^{-1}\\& =(B+J_A)^{-1}+(B+J_A)^{-1}S\big((B+J_A)^{-1}+(B+J)^{-1}S(B+J_A)^{-1}\big).\end{aligned}\]
Thus, we get that

\[\begin{aligned}&(A+I)^{1/2}(B+\sgn(B))^{-1}(A+I)^{1/2}\\
& = L+L(A+I)^{-1/2}S(A+I)^{-1/2}L+ 
L(A+I)^{-1/2}S(B+J)^{-1}S(A+I)^{-1/2}L\end{aligned}\] is bounded. 
By the identity \(B+J_A=(A+I)^{1/2}\widetilde{H}(A+I)^{1/2}\), the operator \[M:=(A+I)^{-1/2}(B+J_A)(A+I)^{-1/2}\] is a bounded operator on its natural domain \((A+I)^{1/2}\dom(B)\). 
Thus \[(A+I)^{-1/2}(B+\sgn(B))(A+I)^{-1/2}=M+(A+I)^{-1/2}(J-J_A)(A+I)^{-1/2}\] is bounded.
\end{proof}   
\end{lemma}
Remark that the operators \eqref{tildeX} and \eqref{tildeY} in the lemma above can be extended to bounded operators. If we consider the same operators, only with 
the absolute value \(|B+\sgn(B)|\) instead of \(B+\sgn(B)\), this extension property is equivalent to the domain stability condition \(\dom(|B|^{1/2})=\dom(A^{1/2})\), see the theorem below.

\begin{theorem}[cf.~ {\cite[Theorem 3.2]{GKMV}} ]\label{equivalent}
Let \(B\) be  the operator associated with the form \(\fb\) in either Theorem \ref{1repgeneral} or \ref{1repoffdiag}. 
Then the following statements are equivalent:
\begin{itemize}

\item[(i)] \quad\(\dom(|B|^{1/2})=\dom(A^{1/2})\),

\item[(ii)] \quad\(\dom(|B|^{1/2})\supseteq\dom(A^{1/2})\),

\item[(ii')] \quad\(\dom(|B|^{1/2})\subseteq\dom(A^{1/2})\),

\item[(iii)] \quad\(X:=(A+I)^{-1/2}|B+\sgn(B)|(A+I)^{-1/2}\) is a bounded symmetric operator on \[\dom(X):=(A+I)^{1/2}\dom(B),\]

\item[(iii')] \quad\(Y:=(A+I)^{1/2}|B+\sgn(B)|^{-1}(A+I)^{1/2}\) is a bounded symmetric operator on \[\dom(Y):=\dom(A^{1/2}),\]

\item[(iv)] \quad\(K:=(A+I)^{1/2}\sgn(B)(A+I)^{-1/2}\) is a bounded involution on \(\cH\),

\item[(v)] \quad\(\sgn(B)\dom(A^{1/2})\subseteq \dom(A^{1/2})\).
\end{itemize}
\begin{proof} 
For brevity, set \(J:=\sgn(B)\). The implication (i) \(\Rightarrow\) (ii) is obvious.  

(ii) \(\Rightarrow\) (iii):\quad Since \(\dom(A^{1/2})\subseteq\dom(|B|^{1/2})=\dom(|B+J|^{1/2})\), we have that the operator \(|B+J|^{1/2}(A+I)^{-1/2}\) has domain \(\cH\) and is thus bounded. Define the positive form 
\[\mathfrak{r}[x,y]:=\langle x,Xy\rangle\quad \text{on}\quad \dom[\mathfrak{r}]:=\dom(X)=(A+I)^{1/2}\dom(B).\]
This form can be represented as a bounded form 
\[\mathfrak{r}[x,y]=\langle |B+J|^{1/2}(A+I)^{-1/2}x,|B+J|^{1/2}(A+I)^{-1/2}y\rangle.\]
Thus,  the associated operator \(X\) is bounded . Note that \(\dom(X)\) is dense by Lemma \ref{tildes}, so that the closure of \(X\) is a bounded operator on \(\cH\).

(ii') \(\Rightarrow\) (iii'):\quad Similarly to the previous implication, the operator \((A+I)^{1/2}|B+J|^{-1/2}\) is bounded and the densely defined positive form 
\[\mathfrak{n}[x,y]:=\langle x,Yy\rangle\quad \text{on}\quad \dom[\mathfrak{n}]:=\dom(Y)=\dom(A)^{1/2}\]
can be represented as a bounded form 
\[\mathfrak{n}[x,y]=\langle (A+I)^{1/2}|B+J|^{-1/2}x,(A+I)^{1/2}|B+J|^{-1/2}y\rangle.\]
Thus, the closure of \(Y\) is a bounded operator on \(\cH\).

(iii) \(\Rightarrow\) (iv):\quad The operator \(K\) is closed on its natural domain \[\dom(K)=\{x\in\cH\;|\;\sgn(B)(A+I)^{-1/2}x\in \dom((A+I)^{1/2})\}.\]
Furthermore, since \(\dom(B)\subseteq \dom(A^{1/2})\) and \(\sgn(B)\) leaves \(\dom(B)\) invariant, we have  \[\dom(X)=(A+I)^{1/2}\dom(B)\subseteq \dom(K).\]

Let \(x\in \dom(X)\), then, taking into account \(\sgn(B+J)=\sgn(B)\), it follows that

\[\begin{aligned}Kx&=(A+I)^{1/2}\sgn(B)(A+I)^{-1/2}x=(A+I)^{1/2}\sgn(B+J)(A+I)^{-1/2}x\\ &=(A+I)^{1/2}(B+J)^{-1}|B+J|(A+I)^{-1/2}x=\big((A+I)^{1/2}(B+J)^{-1}(A+I)^{1/2}\big)Xx.
\end{aligned}\]
By hypothesis and  Lemma \ref{tildes}, respectively, both operators in the product can be extended to bounded operators on \(\cH\). Thus \(K|_{\dom(X)}\) can be boundedly extended to \(\cH\). By the closedness of \(K\), it follows that \(K\) is bounded with \(\dom(K)=\cH\) and it is an involution since \(K^2=I\).

(iii') \(\Rightarrow\) (iv):\quad As in the implication before, \((A+I)^{1/2}\dom(B)\subseteq \dom(K)\) is dense. 

Let \(x\in(A+I)^{1/2}\dom(B)\), then in the same way 
\[\begin{aligned}Kx&=(A+I)^{1/2}\sgn(B+J)(A+I)^{-1/2}x=(A+I)^{1/2}|B+J|^{-1}(B+J)(A+I)^{-1/2}x\\&=\big((A+I)^{1/2}|B+J|^{-1}(A+I)^{1/2}\big)\cdot\big((A+I)^{-1/2}(B+J)(A+I)^{-1/2}\big)x\\&=Y\big((A+I)^{-1/2}(B+J)(A+I)^{-1/2}\big)x
,\end{aligned}\]
where both operators in the product can be boundedly extended to \(\cH\) by Lemma \ref{tildes}. As before, \(K\) is a bounded involution on \(\dom(K)=\cH\).

(iv) \(\Rightarrow\) (v):\quad Since \(\dom(K)=\cH\) by assumption, \(\sgn(B)\) leaves \(\dom(A^{1/2})\) invariant.

(v) \(\Rightarrow\) (i): \quad We first consider the case of strictly positive \(\widetilde{H}\). Then, the positiv-definite form 
\[ \tilde{\fb}[x,y]=\langle (A+I)^{1/2}x,\widetilde{H}(A+I)^{1/2}y\rangle  \quad\text{on }\dom[\tilde{\fb}]=\dom(A^{1/2})\] 
can be represented as \[\tilde{\fb}[x,y]=\langle \widetilde{H}^{1/2}(A+I)^{1/2}x,\widetilde{H}^{1/2}(A+I)^{1/2}y\rangle, \]
so that \(\tilde{\fb}\) is closed by the closedness  the operator \(\widetilde{H}^{1/2}(A+I)^{1/2}\).

The First Representation Theorem for non-negative forms \cite[Theorem VI.2.6]{K} implies the existence of a non-negative operator \(\widetilde{B}\) associated with the form \(\tilde{\fb}\). By construction, we have \(\widetilde{B}=B+J_A\).

The Second Representation Theorem for positive semi-definite quadratic forms \cite[Theorem VI.2.23]{K} yields the equality \(\dom[\tilde{\fb}]=\dom(\widetilde{B}^{1/2})\).

Using the domain equality \eqref{eq:domB}, the claim follows by observing \[\dom[\tilde{\fb}]=\dom[\fb]\quad\text{and}\quad\dom(\widetilde{B}^{1/2})=\dom(B^{1/2}).\]
We now consider the case, where \(\widetilde{H}\) is not necessarily positive.
Define the Hilbert space \[\cH_{A+I}:=\left(\dom(A^{1/2}),\langle (A+I)^{1/2}\;\cdot\;,(A+I)^{1/2}\;\cdot\;\rangle\right).\] Let \(J_{A+I}\) be the operator induced by \(\sgn(B)\) on \(\cH_{A+I}\). The space \(\cH_{A+I}\) is continuously imbedded in \(\cH\) and, by part (v), the operator \(\sgn(B)\) leaves \(\cH_{A+I}\) invariant (as a set). Then, by Lemma \ref{comparison}, the operator \(J_{A+I}\) is continuous on \(\cH_{A+I}\). Since \(J^2=I\), we even have  that \(J_{A+I}\) is a bounded involution, not necessarily unitary. 

It follows that \(K=(A+I)^{1/2}J(A+I)^{-1/2}\) is a bounded involution on \(\cH\).

Observing \(\sgn(B+J)=\sgn(B)\) and \(\dom(B)\subseteq\dom(A^{1/2})\), we have \[\begin{aligned}|B+J|&=(B+J)\sgn(B+J)=(B+J)J\\&= (A+I)^{1/2}(A+I)^{-1/2}(B+J)(A+I)^{-1/2}(A+I)^{1/2}J(A+I)^{-1/2}(A+I)^{1/2}\\&= (A+I)^{1/2}\widetilde{Y}K(A+I)^{1/2},\end{aligned}\]
where we abbreviated \(\widetilde{Y}:=(A+I)^{-1/2}(B+J)(A+I)^{-1/2}\) for the operator in Lemma \ref{tildes}.
Since \(|B+J|\) is non-negative, \(\widetilde{Y}K\) also has to be non-negative. Both \(\widetilde{Y}\) and \(K\) are Hilbert space isomorphisms. Hence, the self-adjoint operator \(\widetilde{Y}K\) has a bounded inverse and is thus strictly positive. Considering the positive form \[\hat{\fb}[x,y]=\langle (A+I)^{1/2}x, (\widetilde{Y}K)(A+I)^{1/2}y\rangle,\quad x,\;y\in \dom(A^{1/2})\] associated with the operator \(|B+J|\), we get from the first case and the functional calculus that \[\dom(A^{1/2})=\dom(|B+J|^{1/2})=\dom(|B|^{1/2})\]
holds. This completes the proof. 
\end{proof}
\end{theorem}
We now give sufficient, but in general not necessary, criteria for the domain stability condition. 
These criteria were introduced in \cite[Lemma 3.6]{GKMV} for the case of strictly positive \(A\). The following lemma shows that they can be extended to the case of non-negative \(A\).

\begin{lemma}[cf. {\cite[Lemma 3.6]{GKMV}}]\label{sufficient}
Let the assumptions of Theorem \(\ref{1repgeneral}\) be satisfied, and let \(B\) be the operator associated with the form \(\fb\). If one of the following conditions 
\begin{itemize}
\item[(a)] the operator \(H\) maps \(\dom(A^{1/2})\) onto itself; 
\item[(b)] the operator \(H\) is strictly positive or strictly negative;
\item[(c)] the operator \(B\) is semi-bounded
\end{itemize} 
holds, then the domain stability condition 
\begin{equation}\label{eq:domcond}\dom(|B|^{1/2})=\dom(A^{1/2})\end{equation} is satisfied.

\begin{proof}
\begin{itemize}
\item[(a)] Since \(H\) is bijective as a map on \(\dom(A^{1/2})\) by assumption, the natural domain of the operator \(B=A^{1/2}HA^{1/2}\) coincides with \(\dom(A)\). Thus, the domains of the positive, boundedly invertible operators \(A+I\) and \(|B|+I\) coincide. The domain stability condition \eqref{eq:domcond} now follows from Corollary \ref{Heinz2} and equation \eqref{dom}. 
\item[(b)] If the operator \(H\) is strictly positive, then \(B\) is non-negative, but may still have a kernel. Choosing \(\sgn(0)=1\), we have \(\sgn(B)=I\). In this case the condition (v) in  Theorem \ref{equivalent} is trivially satisfied. If \(H\) is strictly negative, choose \(\sgn(0)=-1\) so that in this case \(\sgn(B)=-I\) and condition (v) is again satisfied.

\item[(c)]  Without loss of generality, assume \(B\) to be bounded from below.  By Remark \ref{def:B+JA}, the operator \(B+J_A\) is boundedly invertible. Since \(B\) is bounded from below and \(J_A\) is bounded, the operator \(B+J_A+cI\) is strictly positive for all sufficiently large  constants \(c>0\). For such constants, we have the representation
\begin{equation}\label{B+JA}B+J_A+cI=(A+I)^{1/2}(\widetilde{H}+c(A+I)^{-1})(A+I)^{1/2}\end{equation} on \(\dom(B)\subseteq\dom(A^{1/2})\), where \(\widetilde{H}\) is bounded and  boundedly invertible. Since the left hand side of \eqref{B+JA} is a non-negative operator, we have that \(\widetilde{H}+c(A+I)^{-1}\geq 0\) holds. We will show that this operator is even strictly positive. To see this, it suffices to verify that \(\widetilde{H}+c(A+I)^{-1}\) is boundedly invertible. 
Note that \[\widetilde{H}+c(A+I)^{-1}=c\widetilde{H}\left(c^{-1}I+\widetilde{H}^{-1}(A+I)^{-1}\right),\] so this operator is invertible if 

\[-c^{-1}\notin \sigma(\widetilde{H}^{-1}(A+I)^{-1}).\] Recall that for bounded self-adjoint operators \(T_1,T_2\) the well known spectral identity 
\[\sigma(T_1T_2)\setminus \{0\}=\sigma(T_2T_1)\setminus\{0\}\] holds, see, e.g., \cite[Exercise 2.4.11]{Sch}. Thus, the operator has a bounded inverse if \[-c^{-1}\in \rho((A+I)^{-1/2}\widetilde{H}^{-1}(A+I)^{-1/2})=\rho((B+J_A)^{-1}).\] 
Since \(B+J_A\) is, by assumption, bounded from below and boundedly invertible, the negative spectrum of \(B+J_A\) is contained in a bounded interval away from zero, so that \(-c^{-1}\) does not belong to the spectrum if \(c\) is sufficiently large. For those \(c\), the operator \(\widetilde{H}+c(A+I)^{-1}\) is then strictly positive. Form part (b) of the present lemma and \eqref{B+JA}, we deduce that
\[\dom((B+J_A+cI)^{1/2})=\dom((A+I)^{1/2}).\] By Lemma \ref{Heinzapplied} and \eqref{dom}, the equality of the domains of \(A^{1/2}\)  and \(|B|^{1/2}\) holds. \end{itemize}
\end{proof}
\end{lemma}

%%%%%%%%%%%%%%%%%%%%%%%%%%%%%%%%%%%%%%%%%%%%%%%%%%%%%%%%%%%%%

\end{document}